\documentclass{amsart}
\usepackage{amssymb}
\usepackage{amsmath}
\usepackage{amsfonts}

\newtheorem{theorem}{Theorem}
\theoremstyle{plain}

\newtheorem{corollary}{Corollary}

\newtheorem{definition}{Definition}
\newtheorem{example}{Example}

\newtheorem{lemma}{Lemma}

\numberwithin{equation}{section}
\input{tcilatex}

\begin{document}
\title{GENERALIZED $(\tilde{\kappa}\neq -1,\tilde{\mu})$-PARACONTACT\
METRIC\ MANIFOLDS\ WITH\ $\xi (\tilde{\mu})=0$}
\author{I. K\"{u}peli Erken}
\address{Art and Science Faculty,Department of Mathematics, Uludag
University, 16059 Bursa, TURKEY}
\email{iremkupeli@uludag.edu.tr}

\begin{abstract}
We give a local classification of generalized $(\tilde{\kappa}\neq -1,\tilde{%
\mu})$-paracontact metric manifold $(M,\tilde{\varphi},\xi ,\eta ,\tilde{g})$
which satisfies the condition $\xi (\tilde{\mu})=0$. An example of such
manifolds is presented.
\end{abstract}

\maketitle

\section{\textbf{Introduction}}

\label{introduction}

The study of paracontact geometry was introduced by Kaneyuki and Williams in 
\cite{kaneyuki1}. A systematic study of paracontact metric manifolds started
with the paper \cite{Za}, were the Levi-Civita connection, the curvature and
a canonical connection (analogue to the Tanaka Webster connection of the
contact metric case) of a paracontact metric manifold have been described.
However such structures were studied before \cite{RoscVanh}, \cite{BuchRosc}%
, \cite{BuchRosc2}. Note also \cite{Bejan}. These authors called such
structures almost para-coHermitian. The curvature identities for different
classes of almost paracontact metric manifolds were obtained e.g. in \cite%
{DACKO}, \cite{Welyczko}, \cite{Za}. The importance of paracontact geometry,
and in particular of para-Sasakian geometry, has been pointed out especially
in the last years by several papers highlighting the interplays with the
theory of para-K\"{a}hler manifolds and its role in pseudo-Riemannian
geometry and mathematical physics (cf. e.g. \cite{alekseevski1},\cite%
{alekseevski2},\cite{Biz},\cite{cortes1},\cite{cortes2}). Paracontact metric
manifolds have been studied under several different points of view. The case
when the Reeb vector field satisfies a nullity condition was studied in \cite%
{Biz}. The study of three-dimensional paracontact metric $(\tilde{\kappa},%
\tilde{\mu},\tilde{\nu})$-spaces were obtained in \cite{kupmur}.

A remarkable class of paracontact metric manifolds $(M^{2n+1},\tilde{\varphi}%
,\xi ,\eta ,\tilde{g})~$is that of paracontact metric $(\tilde{\kappa},%
\tilde{\mu})$-spaces, which satisfy the nullity condition%
\begin{equation}
\tilde{R}(X,Y)\xi =\tilde{\kappa}\left( \eta \left( Y\right) X-\eta \left(
X\right) Y\right) +\tilde{\mu}(\eta \left( Y\right) \tilde{h}X-\eta \left(
X\right) \tilde{h}Y),  \label{paranullity}
\end{equation}%
for all $X,Y$ vector fields on $M$, where $\tilde{\kappa}$ and $\tilde{\mu}$
are constants and $\tilde{h}=\frac{1}{2}{\mathcal{L}}_{\xi }\tilde{\varphi}$.

This new class of pseudo-Riemannian manifolds was introduced in \cite{MOTE}.
In \cite{Biz}, the authors showed that while the values of $\tilde{\kappa}$
and $\tilde{\mu}$ change the form of (\ref{paranullity}) remains unchanged
under $\mathcal{D}$-homothetic deformations. There are differences between a
contact metric $(\kappa ,\mu )$-space $(M^{2n+1},\varphi ,\xi ,\eta ,g)$ and
a paracontact metric $(\tilde{\kappa},\tilde{\mu})$-space $(M^{2n+1},\tilde{%
\varphi},\xi ,\eta ,\tilde{g})$. Namely, unlike in the contact Riemannian
case, a paracontact $(\tilde{\kappa},\tilde{\mu})$-manifold such that $%
\tilde{\kappa}=-1$ in general is not para-Sasakian. In fact, there are
paracontact $(\tilde{\kappa},\tilde{\mu})$-manifolds such that $\tilde{h}%
^{2}=0$ (which is equivalent to take $\tilde{\kappa}=-1$) but with $\tilde{h}%
\neq 0$. For $5$-dimensional, Cappelletti Montano and Di Terlizzi gave the
first example of paracontact metric $(-1,2)$-space $(M^{2n+1},\tilde{\varphi}%
,\xi ,\eta ,\tilde{g})$ with $\tilde{h}^{2}=0$ but $\tilde{h}\neq 0$ in \cite%
{MOTE} and then Cappelletti Montano et al. gave the first paracontact metric
structures defined on the tangent sphere bundle and constructed an example
with arbitrary $n$ in \cite{Biz}. Later, for $3$-dimensional, the first
numerical example was given in \cite{kupmur}. Another important difference
with the contact Riemannian case, due to the non-positive definiteness of
the metric, is that while for contact metric $(\kappa ,\mu )$-spaces the
constant $\kappa $ can not be greater than $1$, paracontact metric $(\tilde{%
\kappa},\tilde{\mu})$-space has no restriction for the constants $\tilde{%
\kappa}$ and $\tilde{\mu}$.

Koufogiorgos and Tsichlias \cite{KO2} gave a local classification of a
non-Sasakian generalized $(\kappa ,%
\mu
)$-contact metric manifold with $\xi (\mu )=0$. This has been our motivation
for studying generalized $(\tilde{\kappa},\tilde{\mu})$-paracontact metric
manifolds with $\xi (\tilde{\mu})=0$. We would like to emphasize that, as
will be shown in this paper, the class of generalized $(\tilde{\kappa},%
\tilde{\mu})$-paracontact metric manifolds with $\xi (\tilde{\mu})=0$ is
much more different than the class of generalized $(\kappa ,%
\mu
)$-contact metric manifolds with $\xi (\mu )=0$.

By a generalized $(\tilde{\kappa},\tilde{\mu})$-paracontact metric manifold
we mean a $3$-dimensional paracontact metric manifold satisfying (\ref%
{paranullity}) where $\tilde{\kappa}$ and $\tilde{\mu}$ are non constant
smooth functions. In the special case, where $\tilde{\kappa}$ and $\tilde{\mu%
}$ are constant, then $(M^{2n+1},\tilde{\varphi},\xi ,\eta ,\tilde{g})$ is
called a $(\tilde{\kappa},\tilde{\mu})$-paracontact metric manifold.

In \cite{kupmur}, Kupeli Erken and Murathan proved the existence of a new
class of paracontact metric manifolds: the so called $(\tilde{\kappa},\tilde{%
\mu%
},\tilde{\nu})$-paracontact metric manifolds. Such a manifold $M$ is defined
through the condition%
\begin{eqnarray}
\tilde{R}(X,Y)\xi  &=&\tilde{\kappa}\left( \eta \left( Y\right) X-\eta
\left( X\right) Y\right) +\tilde{\mu}(\eta \left( Y\right) \tilde{h}X-\eta
\left( X\right) \tilde{h}Y)  \label{k,m,v} \\
&&+\tilde{\nu}(\eta \left( Y\right) \tilde{\varphi}\tilde{h}X-\eta \left(
X\right) \tilde{\varphi}\tilde{h}Y),  \notag
\end{eqnarray}%
where $\tilde{\kappa},\tilde{%
\mu%
}$ and $\tilde{\nu}$\ are smooth functions on $M.$ Furthermore, it is proved
that these manifolds exist only in the dimension $3$, whereas such a
manifold in dimension greater than $3$ is a $(\tilde{\kappa},\tilde{\mu})$%
-paracontact metric manifold.

The paper is organized in the following way. In Section $2$, we will report
some basic information about paracontact metric manifolds. Some results
about generalized $(\tilde{\kappa}\neq -1,\tilde{\mu})$-paracontact metric
manifolds will be given in Section $3$. In Section $4$, we shall locally
classify generalized $(\tilde{\kappa}\neq -1,\tilde{\mu})$-paracontact
metric manifold with $\xi (\tilde{\mu})=0$ (i.e. the function $\tilde{\mu}$
is constant along the integral curves of the characteristic vector field $%
\xi $). We will prove that we can construct in $R^{3}$ two families of such
manifolds. All manifolds are assumed to be connected.

\section{Preliminaries}

\label{preliminaries}

The aim of this section is to report some basic facts about paracontact
metric manifolds. All manifolds are assumed to be connected and smooth. We
may refer to \cite{kaneyuki1}, \cite{Za} and references therein for more
information about paracontact metric geometry.

An $(2n+1)$-dimensional smooth manifold $M$ is said to have an \emph{almost
paracontact structure} if it admits a $(1,1)$-tensor field $\tilde{\varphi}$%
, a vector field $\xi $ and a $1$-form $\eta $ satisfying the following
conditions:

\begin{enumerate}
\item[(i)] $\eta(\xi )=1$, \ $\tilde{\varphi}^{2}=I-\eta \otimes \xi$,

\item[(ii)] the tensor field $\tilde{\varphi}$ induces an almost paracomplex
structure on each fibre of ${\mathcal{D}}=\ker(\eta)$, i.e. the $\pm 1$%
-eigendistributions, ${\mathcal{D}}^{\pm}:={\mathcal{D}}_{\tilde\varphi}(\pm
1)$ of $\tilde\varphi$ have equal dimension $n$.
\end{enumerate}

From the definition it follows that $\tilde{\varphi}\xi =0$, $\eta \circ 
\tilde{\varphi}=0$ and the endomorphism $\tilde{\varphi}$ has rank $2n$.
When the tensor field $N_{\tilde{\varphi}}:=[\tilde{\varphi},\tilde{\varphi}%
]-2d\eta \otimes \xi $ vanishes identically the almost paracontact manifold
is said to be \emph{normal}. If an almost paracontact manifold admits a
pseudo-Riemannian metric $\tilde{g}$ such that 
\begin{equation}
\tilde{g}(\tilde{\varphi}X,\tilde{\varphi}Y)=-\tilde{g}(X,Y)+\eta (X)\eta
(Y),  \label{G METRIC}
\end{equation}%
for all $X,Y\in \Gamma (TM)$, then we say that $(M,\tilde{\varphi},\xi ,\eta
,\tilde{g})$ is an \textit{almost paracontact metric manifold}. Notice that
any such a pseudo-Riemannian metric is necessarily of signature $(n+1,n)$.
For an almost paracontact metric manifold, there always exists an orthogonal
basis $\{X_{1},\ldots ,X_{n},Y_{1},\ldots ,Y_{n},\xi \}$ such that $\tilde{g}%
(X_{i},X_{j})=\delta _{ij}$, $\tilde{g}(Y_{i},Y_{j})=-\delta _{ij}$ and $%
Y_{i}=\tilde{\varphi}X_{i}$, for any $i,j\in \left\{ 1,\ldots ,n\right\} $.
Such basis is called a $\tilde{\varphi}$-basis.

If in addition $d\eta (X,Y)=\tilde{g}(X,\tilde{\varphi}Y)$ for all vector
fields $X,Y$ on $M,$ $(M,\tilde{\varphi},\xi ,\eta ,\tilde{g})$ is said to
be a \emph{paracontact metric manifold}. In a paracontact metric manifold
one defines a symmetric, trace-free operator $\tilde{h}:=\frac{1}{2}{%
\mathcal{L}}_{\xi }\tilde{\varphi}$. It is known \cite{Za} that $\tilde{h}$
anti-commutes with $\tilde{\varphi}$ and satisfies $\tilde{h}\xi =0,$ tr$%
\tilde{h}=$tr$\tilde{h}\tilde{\varphi}=0$ and 
\begin{equation}
\tilde{\nabla}\xi =-\tilde{\varphi}+\tilde{\varphi}\tilde{h},
\label{nablaxi}
\end{equation}%
where $\tilde{\nabla}$ is the Levi-Civita connection of the
pseudo-Riemannian manifold $(M,\tilde{g})$.

Moreover $\tilde{h}\equiv 0$ if and only if $\xi $ is a Killing vector field
and in this case $(M,\tilde{\varphi},\xi ,\eta ,\tilde{g})$ is said to be a 
\emph{K-paracontact manifold}. A normal paracontact metric manifold is
called a \textit{para-Sasakian manifold}. Also in this context the
para-Sasakian condition implies the $K$-paracontact condition and the
converse holds only in dimension $3$. \ We also recall that any
para-Sasakian manifold satisfies 
\begin{equation}
\tilde{R}(X,Y)\xi =-(\eta (Y)X-\eta (X)Y)  \label{Pasa}
\end{equation}%
so that it is a $(\tilde{\kappa},\tilde{\mu})$-space with $\tilde{\kappa}%
=-1. $ To note that, differently from the contact metric case, condition (%
\ref{Pasa}) is necessary but not sufficient for a paracontact metric
manifold to be para-Sasakian. This fact was already pointed out in \cite{Biz}%
.

As a natural generalization of the above para-Sasakian condition one can
consider contact metric manifolds satisfying (\ref{paranullity}) for some
real numbers $\kappa $ and $\mu .$ Paracontact metric manifolds satisfying (%
\ref{paranullity}) are called $(\tilde{\kappa},\tilde{\mu})$-paracontact
metric manifold. $(\tilde{\kappa},\tilde{\mu})$-paracontact metric manifold%
\textit{\ }were introduced and deeply studied in \cite{MOTE} and \cite{Biz}.

By a generalized $(\tilde{\kappa},\tilde{\mu})$-paracontact metric manifold
we mean a $3$-dimensional paracontact metric manifold satisfying (\ref%
{paranullity}) where $\tilde{\kappa}$ and $\tilde{\mu}$ are non constant
smooth functions. In the special case, where $\tilde{\kappa}$ and $\tilde{\mu%
}$ are constant, then $(M^{2n+1},\tilde{\varphi},\xi ,\eta ,\tilde{g})$ is
called a $(\tilde{\kappa},\tilde{\mu})$-paracontact metric manifold.

Generalized $(\tilde{\kappa},\tilde{\mu})$-paracontact metric manifolds were
studied by Kupeli Erken and Murathan in \cite{kupmur}. A recent
generalization of the $(\tilde{\kappa},\tilde{\mu})$-paracontact metric
manifold is given by the following definition.

\begin{definition}
\label{kmvvv}A $2n+1$-dimensional paracontact metric $(\tilde{\kappa},\tilde{%
\mu},\tilde{\nu})$-manifold is a paracontact metric manifold for which the
curvature tensor field satisfies 
\begin{equation}
\tilde{R}(X,Y)\xi =\tilde{\kappa}(\eta (Y)X-\eta (X)Y)+\tilde{\mu}(\eta (Y)%
\tilde{h}X-\eta (X)\tilde{h}Y)+\tilde{\nu}(\eta (Y)\tilde{\varphi}\tilde{h}%
X-\eta (X)\tilde{\varphi}\tilde{h}Y),  \label{PARAKMUvu}
\end{equation}%
for all $X,Y\in \Gamma (TM)$, where $\tilde{\kappa},\tilde{\mu},\tilde{\nu}$
are smooth functions on $M$.
\end{definition}

A paracontact metric manifold whose characteristic vector field $\xi $ is a
harmonic vector field is called an $H$-paracontact manifold. Moreover,\
Kupeli Erken and Murathan \cite{kupmur} proved \ that $\xi $ is a harmonic
vector field if and only if $\xi $ is an eigenvector of the Ricci operator.
In the same study, they characterized the 3-dimensional $H$-paracontact
metric manifolds in terms of $(\tilde{\kappa},\tilde{\mu},\tilde{\nu})$%
-paracontact metric manifolds. In particular, they proved the following
theorem.

\begin{theorem}
\cite{kupmur} \label{teoh}\textit{Let }$(M,\tilde{\varphi},\xi ,\eta ,\tilde{%
g})$\textit{\ be a }$3$\textit{-dimensional paracontact metric manifold. If
\ the characteristic vector field }$\xi $\textit{\ is harmonic map then the
paracontact metric }$(\tilde{\kappa},\tilde{\mu},\tilde{\nu})$\textit{%
-manifold always exists on every open and dense subset of }$M.$\textit{\
Conversely, if }$M$\textit{\ is a paracontact metric }$(\tilde{\kappa},%
\tilde{\mu},\tilde{\nu})$\textit{-manifold then the characteristic vector
field }$\xi $\textit{\ is harmonic map.}
\end{theorem}

It is shown that condition (\ref{PARAKMUvu}) is meaningless for $\kappa \neq
-1$ in dimension higher than three, because the functions $\tilde{\kappa},%
\tilde{\mu}$ are constants and $\tilde{\nu}$ is the zero function.

\bigskip Given a paracontact metric structure $(\tilde{\varphi},\xi ,\eta ,%
\tilde{g})$ and $\alpha >0$, the change of structure tensors 
\begin{equation}
\bar{\eta}=\alpha \eta ,\text{ \ \ }\bar{\xi}=\frac{1}{\alpha }\xi ,\text{ \
\ }\bar{\varphi}=\tilde{\varphi},\text{ \ \ }\bar{g}=\alpha \tilde{g}+\alpha
(\alpha -1)\eta \otimes \eta  \label{DHOMOTHETIC}
\end{equation}%
is called a \emph{${\mathcal{D}}_{\alpha }$-homothetic deformation}. One can
easily check that the new structure $(\bar{\varphi},\bar{\xi},\bar{\eta},%
\bar{g})$ is still a paracontact metric structure \cite{Za}. We now show
that while ${\mathcal{D}}_{\alpha }$-homothetic deformations destroy
conditions like $\tilde{R}_{XY}\xi =0$, they preserve the class of
paracontact $(\tilde{\kappa},\tilde{\mu})$-spaces.

Kupeli Erken and Murathan analyzed the different possibilities for the
tensor field $\tilde{h}$ in \cite{kupmur}. If $\tilde{h}$ has 
\begin{equation}
\left( 
\begin{array}{ccc}
\tilde{\lambda} & 0 & 0 \\ 
0 & -\tilde{\lambda} & 0 \\ 
0 & 0 & 0%
\end{array}%
\right)  \label{A1}
\end{equation}%
the form (\ref{A1}) respect to local orthonormal $\tilde{\varphi}$-basis $%
\{X,\tilde{\varphi}X,\xi \}$, the authors called the operator $\tilde{h}$ is
of $\mathfrak{h}_{1}$ \textit{type}$.$

If the tensor $\tilde{h}$ has the form $\left( 
\begin{array}{ccc}
0 & 0 & 0 \\ 
1 & 0 & 0 \\ 
0 & 0 & 0%
\end{array}%
\right) $ relative a pseudo orthonormal basis $\{e_{1},e_{2},e_{3}\}$. In
this case, the authors called $\tilde{h}$ is of $\mathfrak{h}_{2}$ \textit{%
type}.

If the matrix form of $\tilde{h}$ is$\ $given$\ $by%
\begin{equation}
\tilde{h}=\left( 
\begin{array}{ccc}
0 & -\tilde{\lambda} & 0 \\ 
\tilde{\lambda} & 0 & 0 \\ 
0 & 0 & 0%
\end{array}%
\right)  \label{A3}
\end{equation}%
with respect to local orthonormal basis $\{X,\tilde{\varphi}X,\xi \}.$ In
this case, the authors said that $\tilde{h}$ is of $\mathfrak{h}_{3}$ 
\textit{type}.

\section{Generalized $(\tilde{\protect\kappa}\neq -1,\tilde{\protect\mu})$%
-paracontact metric manifolds}

\label{Generalized (k,m)-paracontact metric manifolds}

In this section, we will give some basic facts about generalized $(\tilde{%
\kappa}\neq -1,\tilde{\mu})$-paracontact metric manifolds.

\begin{lemma}
\label{L1}Let $(M,\tilde{\varphi},\xi ,\eta ,\tilde{g})~$\ be a generalized $%
(\tilde{\kappa}\neq -1,\tilde{\mu})$-paracontact metric manifold. The
following identities hold:%
\begin{equation}
h^{2}=(1+\tilde{\kappa})\tilde{\varphi}^{2}\text{, \ \ }  \label{hsq}
\end{equation}%
\begin{equation}
\xi (\tilde{\kappa})=0\text{, }  \label{zetak}
\end{equation}%
\begin{equation}
\tilde{Q}\xi =2\tilde{\kappa}\xi ,  \label{Ricci}
\end{equation}%
\begin{equation}
\tilde{Q}=(\frac{\tau }{2}-\tilde{\kappa})I+(-\frac{\tau }{2}+3\tilde{\kappa}%
)\eta \otimes \xi +\tilde{\mu}\tilde{h},\text{ \ \ }\tilde{\kappa}\neq -1,
\label{q1,q3}
\end{equation}%
where $\tilde{Q}$ is the Ricci operator of $M$, $\tau $ denotes scalar
curvature of $M$ and $\tilde{l}=\tilde{R}(.,\xi )\xi $.
\end{lemma}

\begin{proof}
For the proof of (\ref{hsq})-(\ref{Ricci}) are similar to that of \ [\cite%
{kupmur}, Lemma 3.2]. The relation (\ref{q1,q3}) is an immediate consequence
of [\cite{kupmur}, Lemma 4.4 and Lemma 4.14].
\end{proof}

\begin{lemma}
\label{L2}Let $(M,\tilde{\varphi},\xi ,\eta ,\tilde{g})$ be a generalized $(%
\tilde{\kappa}\neq -1,\tilde{\mu})$-paracontact metric manifold. Then, for
any point $P\in M$, with $\tilde{\kappa}(P)>-1$ there exist a neighborhood $U
$ of $P$ and an $\tilde{h}$-frame on $U$, i.e. orthonormal vector fields $%
\xi ,$ $X$, $\tilde{\varphi}X$, defined on $U$, such that 
\begin{equation}
\tilde{h}X=\tilde{\lambda}X,\text{ \ \ }\tilde{h}\tilde{\varphi}X=-\tilde{%
\lambda}\tilde{\varphi}X\text{, \ \ }h\xi =0\text{, \ \ }\tilde{\lambda}=%
\sqrt{1+\tilde{\kappa}}  \label{hfrme}
\end{equation}%
at any point $q\in U$. Moreover, setting $A=X\tilde{\lambda}$ and $B=\tilde{%
\varphi}X\tilde{\lambda}$ on $U$ the following formulas are true :%
\begin{equation}
\tilde{\nabla}_{X}\xi =(\tilde{\lambda}-1)\tilde{\varphi}X,\text{ }\tilde{%
\nabla}_{\tilde{\varphi}X}\xi =-(\tilde{\lambda}+1)X,  \label{eq1}
\end{equation}%
\begin{equation}
\tilde{\nabla}_{\xi }X=-\frac{\tilde{\mu}}{2}\tilde{\varphi}X,\text{ }\tilde{%
\nabla}_{\xi }\tilde{\varphi}X=-\frac{\tilde{\mu}}{2}X,  \label{eq2}
\end{equation}%
\begin{equation}
\tilde{\nabla}_{X}X=-\frac{B}{2\tilde{\lambda}}\tilde{\varphi}X,\text{ \ \ }%
\tilde{\nabla}_{\tilde{\varphi}X}\tilde{\varphi}X=-\frac{A}{2\tilde{\lambda}}%
X,  \label{eq3}
\end{equation}%
\begin{equation}
\tilde{\nabla}_{\tilde{\varphi}X}X=-\frac{A}{2\tilde{\lambda}}\tilde{\varphi}%
X-(\tilde{\lambda}+1)\xi ,\text{ \ \ }\tilde{\nabla}_{X}\tilde{\varphi}X=-%
\frac{B}{2\tilde{\lambda}}X+(1-\tilde{\lambda})\xi ,  \label{eq4}
\end{equation}%
\begin{equation}
\lbrack \xi ,X]=(1-\tilde{\lambda}-\frac{\tilde{\mu}}{2})\tilde{\varphi}X,%
\text{ \ \ }[\xi ,\tilde{\varphi}X]=(\tilde{\lambda}+1-\frac{\tilde{\mu}}{2}%
)X,  \label{eq5}
\end{equation}%
\begin{equation}
\lbrack X,\tilde{\varphi}X]=-\frac{B}{2\tilde{\lambda}}X+\frac{A}{2\tilde{%
\lambda}}\tilde{\varphi}X+2\xi ,  \label{eq6}
\end{equation}%
\begin{equation}
\tilde{h}\text{ }grad\tilde{\mu}=grad\tilde{\kappa},  \label{eq7}
\end{equation}%
\begin{equation}
X\tilde{\mu}=2A,  \label{eq8}
\end{equation}%
\begin{equation}
\tilde{\varphi}X\tilde{\mu}=-2B,  \label{eq9}
\end{equation}%
\begin{equation}
\xi (A)=(1-\tilde{\lambda}-\frac{\tilde{\mu}}{2})B,  \label{eq10}
\end{equation}%
\begin{equation}
\xi (B)=(\tilde{\lambda}+1-\frac{\tilde{\mu}}{2})A,  \label{eq11}
\end{equation}%
\begin{equation}
\left[ \xi ,\tilde{\varphi}grad\tilde{\lambda}\right] =0\text{.}
\label{figrad}
\end{equation}
\end{lemma}

\begin{proof}
The proofs of (\ref{eq1})-(\ref{eq6}) are given in \cite{kupmur}. For the
proof of (\ref{eq7}), we will use well known formula%
\begin{equation*}
\frac{1}{2}grad\text{ }\tau =\sum\limits_{i=1}^{3}\varepsilon _{i}(\tilde{%
\nabla}_{X_{i}}\tilde{Q})X_{i},
\end{equation*}%
where $\{X_{1}=X,$ $X_{2}=\tilde{\varphi}X$, $X_{3}=\xi \}$. Using the
equations (\ref{nablaxi}) and (\ref{L1}), since $tr\tilde{h}=tr\tilde{h}%
\tilde{\varphi}=0$, we obtain%
\begin{eqnarray}
\sum\limits_{i=1}^{3}\varepsilon _{i}(\tilde{\nabla}_{X_{i}}\tilde{Q})X_{i}
&=&\sum\limits_{i=1}^{3}\varepsilon _{i}X_{i}(\frac{\tau }{2}-\tilde{\kappa}%
)X_{i}+\sum\limits_{i=1}^{3}\varepsilon _{i}(X_{i}(\tilde{\mu})\tilde{h}%
\text{ }X_{i})  \notag \\
&&+\tilde{\mu}\sum\limits_{i=1}^{3}\varepsilon _{i}(\nabla _{X_{i}}\tilde{h}%
\text{ })X_{i}  \notag \\
&=&\frac{1}{2}grad\tau -grad\tilde{\kappa}+\tilde{h}\text{ }grad\tilde{\mu}
\label{eq14} \\
&&+\tilde{\mu}\sum\limits_{i=1}^{3}\varepsilon _{i}(\tilde{\nabla}_{X_{i}}%
\tilde{h}\text{ })X_{i}-\frac{1}{2}\xi (\tau )\xi  \notag
\end{eqnarray}
The relations (\ref{hfrme}), (\ref{eq3}) and (\ref{eq4}) yield $%
\sum\limits_{i=1}^{3}(\tilde{\nabla}_{X_{i}}\tilde{h})X_{i}=0.$ Using the
last relation in (\ref{eq14}), one has%
\begin{equation}
\frac{1}{2}grad\text{ }\tau =\frac{1}{2}grad\tau -grad\tilde{\kappa}+\tilde{h%
}\text{ }grad\tilde{\mu}-\frac{1}{2}\xi (\tau )\xi  \label{eq15}
\end{equation}%
that is%
\begin{equation}
-grad\tilde{\kappa}+\tilde{h}\text{ }grad\tilde{\mu}-\frac{1}{2}\xi (\tau
)\xi =0.  \label{eq16}
\end{equation}%
Since the vector field $-grad\tilde{\kappa}+\tilde{h}$ $grad\tilde{\mu}$ is
orthogonal to $\xi .$ So, we get (\ref{eq7}). The equations (\ref{eq8}) and (%
\ref{eq9}) are immediate consequences of (\ref{eq7}).

By virtue of (\ref{zetak}) and (\ref{eq5}), we have%
\begin{eqnarray*}
\xi (A) &=&\xi X\tilde{\lambda}=[\xi ,X]\tilde{\lambda}+X\xi \tilde{\lambda}%
=(1-\tilde{\lambda}-\frac{\tilde{\mu}}{2})\tilde{\varphi}X\tilde{\lambda} \\
\text{ \ \ \ \ \ \ \ \ \ \ \ \ \ \ \ \ \ \ \ \ \ \ \ \ \ \ \ \ \ \ \ \ \ \ \ 
} &=&(1-\tilde{\lambda}-\frac{\tilde{\mu}}{2})B
\end{eqnarray*}%
The relation (\ref{eq11}) is proved similarly. Using (\ref{zetak}), we have 
\begin{equation}
grad\tilde{\lambda}=-AX+B\tilde{\varphi}X,\text{ \ }\tilde{\varphi}grad%
\tilde{\lambda}=-A\tilde{\varphi}X+BX\text{\ .}  \label{grad1}
\end{equation}%
From the relations (\ref{grad1}), (\ref{eq5}), (\ref{eq10}) and (\ref{eq11})
we obtain%
\begin{eqnarray*}
\left[ \xi ,\tilde{\varphi}grad\tilde{\lambda}\right] &=&\left[ \xi ,-A%
\tilde{\varphi}X+BX\right] \\
&=&-(\xi A)\tilde{\varphi}X-A\left[ \xi ,\tilde{\varphi}X\right] +(\xi B)X+B%
\left[ \xi ,X\right] =0.
\end{eqnarray*}
\end{proof}

\begin{lemma}
\label{L3}Let $(M,\tilde{\varphi},\xi ,\eta ,\tilde{g})$ be a generalized $(%
\tilde{\kappa}\neq -1,\tilde{\mu})$-paracontact metric manifold. Then, for
any point $P\in M$, with $\tilde{\kappa}(P)<-1$ there exist a neighborhood $U
$ of $P$ and an $\tilde{h}$-frame on $U$, i.e. orthonormal vector fields $%
\xi ,$ $X$, $\tilde{\varphi}X$, defined on $U$, such that 
\begin{equation}
\tilde{h}X=\tilde{\lambda}\tilde{\varphi}X,\text{ \ \ }\tilde{h}\tilde{%
\varphi}X=-\tilde{\lambda}X\text{, \ \ }h\xi =0\text{, \ \ }\tilde{\lambda}=%
\sqrt{-1-\tilde{\kappa}}  \label{hfrme2}
\end{equation}%
at any point $q\in U$. Moreover, setting $A=X\tilde{\lambda}$ and $B=\tilde{%
\varphi}X\tilde{\lambda}$ on $U$ the following formulas are true :%
\begin{equation}
\tilde{\nabla}_{X}\xi =-\tilde{\varphi}X+\tilde{\lambda}X,\text{ }\tilde{%
\nabla}_{\tilde{\varphi}X}\xi =-X-\tilde{\lambda}\tilde{\varphi}X,
\label{eq1.1}
\end{equation}%
\begin{equation}
\tilde{\nabla}_{\xi }X=-\frac{\tilde{\mu}}{2}\tilde{\varphi}X,\text{ }\tilde{%
\nabla}_{\xi }\tilde{\varphi}X=-\frac{\tilde{\mu}}{2}X,  \label{eq2.2}
\end{equation}%
\begin{equation}
\tilde{\nabla}_{X}X=-\frac{B}{2\tilde{\lambda}}\tilde{\varphi}X+\tilde{%
\lambda}\xi ,\text{ \ \ }\tilde{\nabla}_{\tilde{\varphi}X}\tilde{\varphi}X=-%
\frac{A}{2\tilde{\lambda}}X+\tilde{\lambda}\xi ,  \label{eq3.3}
\end{equation}%
\begin{equation}
\tilde{\nabla}_{\tilde{\varphi}X}X=-\frac{A}{2\tilde{\lambda}}\tilde{\varphi}%
X-\xi ,\text{ \ \ }\tilde{\nabla}_{X}\tilde{\varphi}X=-\frac{B}{2\tilde{%
\lambda}}X+\xi ,  \label{eq44}
\end{equation}%
\begin{equation}
\lbrack \xi ,X]=-\tilde{\lambda}X+(1-\frac{\tilde{\mu}}{2})\tilde{\varphi}X,%
\text{ \ \ }[\xi ,\tilde{\varphi}X]=(1-\frac{\tilde{\mu}}{2})X+\tilde{\lambda%
}\tilde{\varphi}X,  \label{eq5.5}
\end{equation}%
\begin{equation}
\lbrack X,\tilde{\varphi}X]=-\frac{B}{2\tilde{\lambda}}X+\frac{A}{2\tilde{%
\lambda}}\tilde{\varphi}X+2\xi ,  \label{eq6.6}
\end{equation}%
\begin{equation}
\tilde{h}\text{ }grad\tilde{\mu}=grad\tilde{\kappa},  \label{eq7.7}
\end{equation}%
\begin{equation}
X\tilde{\mu}=2B,  \label{eq8.8}
\end{equation}%
\begin{equation}
\tilde{\varphi}X\tilde{\mu}=-2A,  \label{eq9.9}
\end{equation}%
\begin{equation}
\xi (A)=-\tilde{\lambda}A+(1-\frac{\tilde{\mu}}{2})B,  \label{eq10.10}
\end{equation}%
\begin{equation}
\xi (B)=(1-\frac{\tilde{\mu}}{2})A+\tilde{\lambda}B,  \label{eq11.11}
\end{equation}%
\begin{equation}
\left[ \xi ,\tilde{\varphi}grad\tilde{\lambda}\right] =0\text{.}
\label{figrad1}
\end{equation}
\end{lemma}

\begin{proof}
The proofs of (\ref{eq1.1})-(\ref{eq6.6}) are given in \cite{kupmur}. The
proof of (\ref{eq7.7}) is similar to proof of Lemma \ref{L2}, equation (\ref%
{eq7}). The equations (\ref{eq8.8}) and (\ref{eq9.9}) are immediate
consequences of (\ref{eq7.7}).

By virtue of (\ref{zetak}) and (\ref{eq5.5}), we have%
\begin{eqnarray*}
\xi (A) &=&\xi X\tilde{\lambda}=[\xi ,X]\tilde{\lambda}+X\xi \tilde{\lambda}%
=-\tilde{\lambda}X\tilde{\lambda}+(1-\frac{\tilde{\mu}}{2})\tilde{\varphi}X%
\tilde{\lambda} \\
\text{ \ \ \ \ \ \ \ \ \ \ \ \ \ \ \ \ \ \ \ \ \ \ \ \ \ \ \ \ \ \ \ \ \ \ \ 
} &=&-\tilde{\lambda}A+(1-\frac{\tilde{\mu}}{2})B
\end{eqnarray*}%
The relation (\ref{eq11.11}) is proved similarly. Using (\ref{zetak}), we
have 
\begin{equation}
grad\tilde{\lambda}=-AX+B\tilde{\varphi}X,\text{ \ }\tilde{\varphi}grad%
\tilde{\lambda}=-A\tilde{\varphi}X+BX\text{\ .}  \label{grad1.1}
\end{equation}%
From the relations (\ref{grad1.1}), (\ref{eq5.5}), (\ref{eq10.10}) and (\ref%
{eq11.11}) we obtain%
\begin{eqnarray*}
\left[ \xi ,\tilde{\varphi}grad\tilde{\lambda}\right] &=&\left[ \xi ,-A%
\tilde{\varphi}X+BX\right] \\
&=&-(\xi A)\tilde{\varphi}X-A\left[ \xi ,\tilde{\varphi}X\right] +(\xi B)X+B%
\left[ \xi ,X\right] =0.
\end{eqnarray*}
\end{proof}

\section{Generalized $(\tilde{\protect\kappa}\neq -1,\tilde{\protect\mu})$%
-paracontact metric manifolds with $\protect\xi (\tilde{\protect\mu})=0$}

\label{(k,m)-paracontact metric with}

\bigskip We shall give a local classification of generalized $(\tilde{\kappa}%
\neq -1,\tilde{\mu})$-paracontact metric manifolds with $\tilde{\kappa}>-1$
which satisfy the condition $\xi (\tilde{\mu})=0$.

\begin{theorem}[Main Theorem]
\label{AX} Let $(M,\tilde{\varphi},\xi ,\eta ,\tilde{g})$ be a generalized $(%
\tilde{\kappa}\neq -1,\tilde{\mu})$-paracontact metric manifold with $\tilde{%
\kappa}>-1$ and $\xi (\tilde{\mu})=0$. Then

$1)$ At any point of $M$, precisely one of the following relations is valid: 
$\tilde{\mu}=2(1+\sqrt{1+\tilde{\kappa}}),$ or $\tilde{\mu}=2(1-\sqrt{1+%
\tilde{\kappa}})$

$2)$ At any point $P\in M$ there exists a chart $(U,(x,y,z))$ with $P\in
U\subseteq M,$ such that

\ \ \ \ \ \ $i)$ the functions $\tilde{\kappa},\tilde{\mu}$ depend only on
the variable $z.$

\ \ \ \ \ \ $ii)$ if $\tilde{\mu}=2(1+\sqrt{1-\tilde{\kappa}}),$ $($resp. $%
\tilde{\mu}=2(1-\sqrt{1-\tilde{\kappa}})),$ the tensor fields $\eta $, $\xi $%
, $\tilde{\varphi}$, $\tilde{g}$, $\tilde{h}$ are given by the relations,%
\begin{equation*}
\xi =\frac{\partial }{\partial x},\text{ \ \ }\eta =dx-adz
\end{equation*}%
\begin{equation*}
\tilde{g}=\left( 
\begin{array}{ccc}
1 & 0 & -a \\ 
0 & 1 & -b \\ 
-a & -b & -1+a^{2}+b^{2}%
\end{array}%
\right) \text{ \ \ \ \ }\left( \text{resp. \ \ }\tilde{g}=\left( 
\begin{array}{ccc}
1 & 0 & -a \\ 
0 & -1 & -b \\ 
-a & -b & 1+a^{2}+b^{2}%
\end{array}%
\right) \right) ,
\end{equation*}%
\begin{equation*}
\tilde{\varphi}=\left( 
\begin{array}{ccc}
0 & a & -ab \\ 
0 & b & 1-b^{2} \\ 
0 & 1 & -b%
\end{array}%
\right) \text{ \ \ \ \ }\left( \text{resp. \ \ }\tilde{\varphi}=\left( 
\begin{array}{ccc}
0 & a & -ab \\ 
0 & b & 1-b^{2} \\ 
0 & 1 & -b%
\end{array}%
\right) \right) ,
\end{equation*}%
\begin{equation*}
\tilde{h}=\left( 
\begin{array}{ccc}
0 & 0 & a\tilde{\lambda} \\ 
0 & -\tilde{\lambda} & 2\tilde{\lambda}b \\ 
0 & 0 & \tilde{\lambda}%
\end{array}%
\right) \text{ \ \ \ \ \ }\left( \text{resp. \ \ }\tilde{h}=\left( 
\begin{array}{ccc}
0 & 0 & -a\tilde{\lambda} \\ 
0 & \tilde{\lambda} & -2\tilde{\lambda}b \\ 
0 & 0 & -\tilde{\lambda}%
\end{array}%
\right) \right)
\end{equation*}%
with respect to the basis $\left( \frac{\partial }{\partial x},\frac{%
\partial }{\partial y},\frac{\partial }{\partial z}\right) ,$ where $%
a=-2y+f(z)$ \ \ (resp.\ $a=2y+f(z)$), $b=-\frac{y}{2}\frac{r^{^{\prime }}(z)%
}{r(z)}-2xr(z)+s(z)$,\ $\tilde{\lambda}=\tilde{\lambda}(z)=r(z)$ \ and $f(z)$%
, $r(z)$, $s(z)$ are arbitrary smooth functions of $z.$
\end{theorem}

\begin{proof}[Proof of the Main Theorem:]
Let $\left\{ \xi ,X,\tilde{\varphi}X\right\} $ be an $\tilde{h}$-frame, such
that%
\begin{equation*}
\tilde{h}X=\tilde{\lambda}X,\text{ \ \ }\tilde{h}\tilde{\varphi}X=-\tilde{%
\lambda}\tilde{\varphi}X,\text{ \ \ \ }\tilde{\lambda}=\sqrt{1+\tilde{\kappa}%
}
\end{equation*}%
in an appropriate neighbourhood of an arbitrary point of $M$. Using the
hypothesis $\xi (\tilde{\mu})=0$ and equations (\ref{eq8})-(\ref{figrad})
and (\ref{grad1}) we have the following relations,%
\begin{equation}
(\tilde{\varphi}grad\lambda )\tilde{\mu}=4AB,  \label{eqlamda1}
\end{equation}%
\begin{equation}
\left[ \xi ,\tilde{\varphi}grad\lambda \right] \tilde{\mu}=0,
\label{eqlamda2}
\end{equation}%
\begin{equation}
\xi (AB)=0,  \label{zeta(AB)}
\end{equation}%
\begin{equation}
A\xi B+B\xi A=0,  \label{zeta AB1}
\end{equation}%
\begin{equation}
A^{2}(\tilde{\lambda}+1-\frac{\tilde{\mu}}{2})+B^{2}(1-\tilde{\lambda}-\frac{%
\tilde{\mu}}{2})=0\text{.}  \label{zetaAB2}
\end{equation}

Differentiating the relation (\ref{zetaAB2}) with respect to $\xi $ and
using the equations (\ref{zetak}), \ $\xi (\tilde{\mu})=0$, (\ref{eq10}), (%
\ref{eq11}) and (\ref{zetaAB2}), we obtain 
\begin{equation}
(1+\tilde{\lambda}-\frac{\tilde{\mu}}{2})(-\tilde{\lambda}+1-\frac{\tilde{\mu%
}}{2})AB=0.  \label{eqlamda3}
\end{equation}

We put $F=(1+\tilde{\lambda}-\frac{\tilde{\mu}}{2})(-\tilde{\lambda}+1-\frac{%
\tilde{\mu}}{2})$ \ and consider the set $N=\left\{ p\in M\mid (\func{grad}%
\tilde{\lambda})(p)\neq 0\right\} .$ We will prove that $F=0$ at any point
of $N$. Let $p\in N$ be such that $F(p)\neq 0$. From (\ref{eqlamda3}) we get 
$(AB)(p)=0$. We consider cases $\{A(p)=$ $B(p)=0$ $\},\{A(p)\neq 0$, $%
B(p)=0\}$ and $\{A(p)=0$, $B(p)\neq 0\}$. Now we will examine the first
case. In this case, by (\ref{zetak}), we get $(\xi (\tilde{\lambda}))(p)=0.$
As a result we obtain $(\func{grad}\tilde{\lambda})(p)=0$ which is a
contradiction with $(\func{grad}\tilde{\lambda})(p)\neq 0.$ So, the first
case is impossible. We assume that $\{A(p)\neq 0$, $B(p)=0\}$. Since the
function $F$ is continuous, we find that a neighbourhood $V\subseteq N$
exists, with $p\in V$ such that $F\neq 0$ at any point of $V$. Similarly,
due to the fact that the function $A$ is continuous on its domain, a
neighbourhood $W$ of $p$ exists with $p\in W\subset V$, such that $A\neq 0$
at any point of $W$, and thus $B=0$ on $W$. From (\ref{zetaAB2}), we have $%
(1-\tilde{\lambda}-\frac{\tilde{\mu}}{2})=0$ at any point of $W$ and thus $%
F=0$ on $W$, which is a contradiction. Since the last case is similar to the
second case we omit it. Therefore, $F=0$ at any point of $N$. In what
follows, we will work on the complement $N^{C}$ of set $N$, in order to
prove that $F=0$ on $M$. If $N^{C}=\varnothing $, then $F=0$ on $M$. \ Let
us suppose that $N^{C}\neq \varnothing $. Then $\ $we have $\func{grad}%
\tilde{\lambda}=0$ on $N^{C}$ and thus the function of $\tilde{\lambda}$ is
constant at any connected component of the interior ($N^{C})^{\circ }$. From
the constancy of $\tilde{\lambda}$ and the relations (\ref{eq8}) and (\ref%
{eq9}), $\xi (\tilde{\mu})=0$, the function $\tilde{\mu}$ is also constant.
As a result we find that $F$ is constant on any connected component of ($%
N^{C})^{\circ }$. Because $M$ is connected and $F=0$ on $N$ and $F=$
constant on any connected component of ($N^{C})^{\circ }$ we conclude that $%
F=0,$ or equivalently $(1+\tilde{\lambda}-\frac{\tilde{\mu}}{2})(-\tilde{%
\lambda}+1-\frac{\tilde{\mu}}{2})=0$ at any point of $M$.

Now we consider the open disjoint sets $U_{0}=\{p\in M\mid (\tilde{\lambda}%
+1-\frac{\tilde{\mu}}{2})(p)\neq 0$ $\}$ and $U_{1}=\{p\in M\mid (1-\tilde{%
\lambda}-\frac{\tilde{\mu}}{2})(p)\neq 0$ $\}$.We have $U_{0}$ $\cup $ $%
U_{1}=M.$ Due to the fact that $M$ is connected, we conclude that $\{M=U_{0}$
and $U_{1}=\varnothing \}$ or $\{U_{0}=\varnothing $ and $U_{1}=M\}$.
Regarding the set $U_{0}$ we have $\tilde{\mu}=2(1+\tilde{\lambda})$, or
equivalently $\tilde{\mu}=2(1+\sqrt{1+\tilde{\kappa}})$ at any point $M$.
Similarly, regarding the set $U_{1}$ we obtain $\tilde{\mu}=2(1-\tilde{%
\lambda})=2(1-\sqrt{1+\tilde{\kappa})}$. Therefore, $(1)$ is proved. Now, we
will examine the cases $\tilde{\mu}=2(1+\sqrt{1+\tilde{\kappa}})$ and $%
\tilde{\mu}=2(1-\sqrt{1+\tilde{\kappa}})$.

\textit{Case 1.} $\tilde{\mu}=2(1+\sqrt{1+\tilde{\kappa}}).$

Let $p\in M$ and $\{\xi ,X,\tilde{\varphi}X\}$ be an $\tilde{h}$-frame on a
neighborhood $U$ of $p.$ Using the assumption $\tilde{\mu}=2(1+\sqrt{1+%
\tilde{\kappa}})$ and (\ref{zetaAB2}) we obtain $B=0$ and thus the relations
(\ref{eq5}) and (\ref{eq6}) are reduced to%
\begin{equation}
\lbrack \xi ,X]=-2\tilde{\lambda}\tilde{\varphi}X,  \label{eq4.1}
\end{equation}%
\begin{equation}
\text{\ }[\xi ,\tilde{\varphi}X]=0,\text{ \ }  \label{eq4.2}
\end{equation}%
\begin{equation}
\text{\ }[X,\tilde{\varphi}X]=-\frac{A}{2\tilde{\lambda}}\tilde{\varphi}%
X+2\xi .  \label{eq4.3}
\end{equation}%
Since $[\xi ,\tilde{\varphi}X]=0$, the distribution which is spanned by\ $%
\xi $ and $\tilde{\varphi}X$ is integrable and so for any $q\in V$ there
exist a chart $(V,(x,y,z)\}$ at $p\in V\subset U$, such that 
\begin{equation}
\xi =\frac{\partial }{\partial x},\text{ \ \ }X=a\frac{\partial }{\partial x}%
+b\frac{\partial }{\partial y}+c\frac{\partial }{\partial z},\text{ \ \ }%
\tilde{\varphi}X=\frac{\partial }{\partial y},  \label{eq4.4}
\end{equation}%
where $a$, $b$ and $c$ are smooth functions on $V$. Since $\xi $, $X$ and $%
\tilde{\varphi}X$ are linearly independent we have $c\neq 0$ at any point of 
$V$. By using (\ref{eq4.4}), (\ref{zetak}) and $B=0$ we obtain%
\begin{equation}
\frac{\partial \tilde{\lambda}}{\partial x}=0\text{ \ \ and \ }\frac{%
\partial \tilde{\lambda}}{\partial y}=0\text{\ .}  \label{eq4.5}
\end{equation}%
From (\ref{eq4.5}) we find 
\begin{equation}
\tilde{\lambda}=r(z),  \label{eqlamda}
\end{equation}%
where $r(z)$ is smooth function of $z$ defined on $V$. By using (\ref{eq4.1}%
), (\ref{eq4.3}) and (\ref{eq4.4}) we have following partial differential
equations:%
\begin{equation}
\frac{\partial a}{\partial x}=0,\text{ \ }\frac{\partial b}{\partial x}=-2%
\tilde{\lambda},\text{ \ \ }\frac{\partial c}{\partial x}=0,  \label{eq4.6}
\end{equation}%
\begin{equation}
\frac{\partial a}{\partial y}=-2,\text{ \ }\frac{\partial b}{\partial y}=-%
\frac{A}{2\tilde{\lambda}},\text{ \ \ }\frac{\partial c}{\partial y}=0.
\label{eq4.7}
\end{equation}%
From $\frac{\partial c}{\partial x}=\frac{\partial c}{\partial y}=0$ it
follows that $c=c(z)$ and because of the fact that $c\neq 0$, we can assume
that $c=1$ through a reparametrization of the variable $z$. For the sake of
simplicity we will continue to use the same coordinates $(x,y,z),$ taking
into account that $c=1$ in the relations that we have occured. From $\frac{%
\partial a}{\partial x}=0,$ $\frac{\partial a}{\partial y}=-2$ we obtain 
\begin{equation*}
a=a(x,y,z)=-2y+f(z),
\end{equation*}%
where $f(z)$ is smooth function of $z$ defined on $V$. Differentiating $%
\tilde{\lambda}$ with respect to $X$ and using (\ref{eq4.5}) and (\ref%
{eqlamda}) we have%
\begin{equation}
A=r^{\prime }(z),  \label{eq4.8}
\end{equation}%
where $r^{\prime }(z)=\frac{dr}{dz}$. By using the relations $\frac{\partial
b}{\partial x}=-2\tilde{\lambda}$, $\frac{\partial b}{\partial y}=-\frac{A}{2%
\tilde{\lambda}}$ and (\ref{eqlamda}) we get 
\begin{equation*}
b=-\frac{y}{2}\frac{r^{\prime }(z)}{r(z)}-2xr(z)+s(z).
\end{equation*}%
where $s(z)$ is arbitrary smooth function of $z$ defined on $V$. We will
calculate the tensor fields $\eta $, $\tilde{\varphi}$, $\tilde{g}$ and $%
\tilde{h}$ with respect to the basis $\frac{\partial }{\partial x}$, $\frac{%
\partial }{\partial y}$, $\frac{\partial }{\partial z}$. For the components $%
\tilde{g}_{ij}$ of the Riemannian metric, using (\ref{eq4.4}) we have 
\begin{equation*}
\tilde{g}_{11}=\tilde{g}(\frac{\partial }{\partial x},\frac{\partial }{%
\partial x})=1,\text{ }\tilde{g}(\xi ,\xi )=1,\text{ \ }\tilde{g}_{22}=%
\tilde{g}(\frac{\partial }{\partial y},\frac{\partial }{\partial y})=\tilde{g%
}(\tilde{\varphi}X,\tilde{\varphi}X)=1,\text{ }
\end{equation*}%
\begin{equation*}
\text{\ \ }\tilde{g}_{12}=\tilde{g}_{21}=\tilde{g}(\frac{\partial }{\partial
x},\frac{\partial }{\partial y})=0,
\end{equation*}%
\begin{eqnarray*}
\tilde{g}_{13} &=&\tilde{g}_{31}=\tilde{g}(\frac{\partial }{\partial x},X-a%
\frac{\partial }{\partial x}-b\frac{\partial }{\partial y}) \\
&=&\tilde{g}(\xi ,X)-a\tilde{g}_{11}=-a,
\end{eqnarray*}%
\begin{eqnarray*}
\tilde{g}_{23} &=&\tilde{g}_{32}=\tilde{g}(\frac{\partial }{\partial y},X-a%
\frac{\partial }{\partial x}-b\frac{\partial }{\partial y}) \\
&=&\tilde{g}(\tilde{\varphi}X,X)-a\tilde{g}_{12}-b\tilde{g}_{22}=-b,
\end{eqnarray*}%
\begin{eqnarray*}
-1 &=&\tilde{g}(X,X)\Rightarrow a^{2}\tilde{g}_{11}+2a\tilde{g}_{13}+b^{2}%
\tilde{g}_{22}+2ab\tilde{g}_{12}+2b\tilde{g}_{23}+\tilde{g}_{33}=-1 \\
&=&a^{2}-2a^{2}+b^{2}-2b^{2}+\tilde{g}_{33}=\tilde{g}_{33}-a^{2}-b^{2},
\end{eqnarray*}%
from which we obtain $\tilde{g}_{33}=-1+a^{2}+b^{2}$.

The matrix form of $\tilde{g}$ is$\ $given$\ $by%
\begin{equation*}
\tilde{g}=\left( 
\begin{array}{ccc}
1 & 0 & -a \\ 
0 & 1 & -b \\ 
-a & -b & -1+a^{2}+b^{2}%
\end{array}%
\right) .
\end{equation*}

The components of the tensor field $\tilde{\varphi}$ are immediate
consequences of 
\begin{equation*}
\tilde{\varphi}(\xi )=\tilde{\varphi}(\frac{\partial }{\partial x})=0,\text{
\ \ }\tilde{\varphi}(\frac{\partial }{\partial y})=X=a\frac{\partial }{%
\partial x}+b\frac{\partial }{\partial y}+\frac{\partial }{\partial z},
\end{equation*}%
\begin{eqnarray*}
\tilde{\varphi}(\frac{\partial }{\partial z}) &=&\tilde{\varphi}(X-a\frac{%
\partial }{\partial x}-b\frac{\partial }{\partial y})=\tilde{\varphi}X-a%
\tilde{\varphi}(\frac{\partial }{\partial x})-b\tilde{\varphi}(\frac{%
\partial }{\partial y}) \\
&=&\tilde{\varphi}X-b(a\frac{\partial }{\partial x}+b\frac{\partial }{%
\partial y}+\frac{\partial }{\partial z}) \\
&=&\frac{\partial }{\partial y}-ab\frac{\partial }{\partial x}-b^{2}\frac{%
\partial }{\partial y}-b\frac{\partial }{\partial z} \\
&=&-ab\frac{\partial }{\partial x}+(1-b^{2})\frac{\partial }{\partial y}-b%
\frac{\partial }{\partial z}.
\end{eqnarray*}%
The matrix form of $\tilde{\varphi}$ is$\ $given$\ $by%
\begin{equation*}
\tilde{\varphi}=\left( 
\begin{array}{ccc}
0 & a & -ab \\ 
0 & b & 1-b^{2} \\ 
0 & c & -b%
\end{array}%
\right) .
\end{equation*}%
The expression of the 1-form $\eta $, immediately follows from $\eta (\xi
)=1 $, $\eta (X)=\eta (\tilde{\varphi}X)=0$%
\begin{equation*}
\eta =dx-adz.
\end{equation*}%
Now we calculate the components of the tensor field $\tilde{h}$ with respect
to the basis $\frac{\partial }{\partial x}$, $\frac{\partial }{\partial y}$, 
$\frac{\partial }{\partial z}$.%
\begin{equation*}
\tilde{h}(\xi )=\tilde{h}(\frac{\partial }{\partial x})=0,\text{ \ \ }\tilde{%
h}(\frac{\partial }{\partial y})=-\lambda \frac{\partial }{\partial y},
\end{equation*}%
\begin{eqnarray*}
\tilde{h}(\frac{\partial }{\partial z}) &=&\tilde{h}(X-a\frac{\partial }{%
\partial x}-b\frac{\partial }{\partial y}) \\
&=&\tilde{h}X-a\tilde{h}(\frac{\partial }{\partial x})-b\tilde{h}(\frac{%
\partial }{\partial y}) \\
&=&\tilde{\lambda}X+b\tilde{\lambda}\frac{\partial }{\partial y} \\
&=&\tilde{\lambda}(a\frac{\partial }{\partial x}+b\frac{\partial }{\partial y%
}+\frac{\partial }{\partial z})+b\tilde{\lambda}\frac{\partial }{\partial y},
\end{eqnarray*}%
\begin{equation*}
\tilde{h}(\frac{\partial }{\partial z})=\tilde{\lambda}a\frac{\partial }{%
\partial x}+2b\tilde{\lambda}\frac{\partial }{\partial y}+\tilde{\lambda}%
\frac{\partial }{\partial z}.
\end{equation*}%
The matrix form of $\tilde{h}$ is$\ $given$\ $by%
\begin{equation*}
\tilde{h}=\left( 
\begin{array}{ccc}
0 & 0 & a\tilde{\lambda} \\ 
0 & -\tilde{\lambda} & 2\tilde{\lambda}b \\ 
0 & 0 & \tilde{\lambda}%
\end{array}%
\right) .
\end{equation*}%
Thus the proof of the Case 1 is completed.

\textit{Case 2.} $\tilde{\mu}=2(1-\sqrt{1+\tilde{\kappa}}).$

As in the Case 1, we consider an $\tilde{h}$-frame $\{\xi ,X,\tilde{\varphi}%
X\}$. Using the assumption $\tilde{\mu}=2(1-\sqrt{1+\tilde{\kappa}})$ and (%
\ref{zetaAB2}) we obtain $A=0$ and thus the relation (\ref{eq5}) and (\ref%
{eq6}) is written as%
\begin{equation}
\lbrack \xi ,X]=0,  \label{eq4.9}
\end{equation}%
\begin{equation}
\text{\ }[\xi ,\tilde{\varphi}X]=2\tilde{\lambda}X,\text{ \ }  \label{eq4.10}
\end{equation}%
\begin{equation}
\text{\ }[X,\tilde{\varphi}X]=-\frac{B}{2\tilde{\lambda}}X+2\xi .
\label{eq4.11}
\end{equation}%
Because of (\ref{eq4.9}) we find that there is a chart $(V^{\prime
},(x,y,z)) $ such that 
\begin{equation*}
\xi =\frac{\partial }{\partial x},\text{ \ \ \ }X=\frac{\partial }{\partial y%
}
\end{equation*}%
on $V^{\prime }$. We put%
\begin{equation*}
\tilde{\varphi}X=a\frac{\partial }{\partial x}+b\frac{\partial }{\partial y}%
+c\frac{\partial }{\partial z},
\end{equation*}%
where $a,b,c$ are smooth functions defined on $V^{\prime }.$ As in the Case
1, we can directly calculate the tensor fields $\eta $, $\tilde{\varphi}$, $%
\tilde{g}$ and $\tilde{h}$ with respect to the basis $\frac{\partial }{%
\partial x}$, $\frac{\partial }{\partial y}$, $\frac{\partial }{\partial z}$%
. \ This completes the proof of the main theorem.
\end{proof}

Now, we give an example of a generalized $(\tilde{\kappa}\neq -1,\tilde{\mu})
$-paracontact metric manifold with $\xi (\tilde{\mu})=0$ which satisfy the
conditions of Main Theorem (Case 1).

\begin{example}
\label{ex1} We consider the $3$-dimensional manifold%
\begin{equation*}
M=\{(x,y,z)\in R^{3},z\neq 0\}
\end{equation*}%
and the vector fields%
\begin{equation*}
\xi =\text{ }\frac{\partial }{\partial x},\text{ \ \ }\tilde{\varphi}X=\frac{%
\partial }{\partial y},\text{ \ \ }X=(-2y+1)\frac{\partial }{\partial x}+(-%
\frac{y}{2z}-2xz+2)\frac{\partial }{\partial y}+\frac{\partial }{\partial z}.
\end{equation*}%
The 1-form $\eta =dx-(-2y+1)dz$ defines a contact structure on $M$ with
characteristic vector field $\xi =\frac{\partial }{\partial x}$. Let $\tilde{%
g}$, $\tilde{\varphi}$ be the pseudo-Riemannian metric and the $(1,1)$%
-tensor field given by 
\begin{eqnarray*}
\tilde{g} &=&\left( 
\begin{array}{ccc}
1 & 0 & 2y-1 \\ 
0 & 1 & \frac{y}{2z}+2xz-2 \\ 
2y-1 & \frac{y}{2z}+2xz-2 & -1+(-2y+1)^{2}+(-\frac{y}{2z}-2xz+2)^{2}%
\end{array}%
\right) ,\text{ } \\
\tilde{\varphi}\text{\ } &=&\left( 
\begin{array}{ccc}
0 & -2y+1 & -(-2y+1)(-\frac{y}{2z}-2xz+2) \\ 
0 & -\frac{y}{2z}-2xz+2 & 1-(-\frac{y}{2z}-2xz+2)^{2} \\ 
0 & 1 & \frac{y}{2z}+2xz-2%
\end{array}%
\right) , \\
\text{\ }\tilde{h}\text{\ } &=&\left( 
\begin{array}{ccc}
0 & 0 & (-2y+1)z \\ 
0 & -z & 2z(-\frac{y}{2z}-2xz+2) \\ 
0 & 0 & z%
\end{array}%
\right) ,\text{ \ \ }\tilde{\lambda}=z,
\end{eqnarray*}%
with respect to the basis $\frac{\partial }{\partial x},\frac{\partial }{%
\partial y},\frac{\partial }{\partial z}$.${\small \ }$
\end{example}

Let $\left\{ \xi ,X,\tilde{\varphi}X\right\} $ be an $\tilde{h}$-frame, such
that%
\begin{equation*}
\tilde{h}X=\tilde{\lambda}\tilde{\varphi}X,\text{ \ \ }\tilde{h}\tilde{%
\varphi}X=-\tilde{\lambda}X,\text{ \ \ \ }\tilde{\lambda}=\sqrt{-1-\tilde{%
\kappa}}
\end{equation*}%
in an appropriate neighbourhood of an arbitrary point of $M$. Using the
hypothesis $\xi (\tilde{\mu})=0$ and equations (\ref{eq8.8})-(\ref{figrad1})
and (\ref{grad1.1}) we have the following relations,%
\begin{equation}
(\tilde{\varphi}grad\lambda )\tilde{\mu}=2(A^{2}+B^{2}),  \label{eqlamda1.1}
\end{equation}%
\begin{equation}
\left[ \xi ,\tilde{\varphi}grad\lambda \right] \tilde{\mu}=0,
\label{eqlamda2.2}
\end{equation}%
\begin{equation}
\xi (A^{2}+B^{2})=0,  \label{zeta(AB)B}
\end{equation}%
\begin{equation}
A\xi A+B\xi B=0,  \label{zeta AB1.1}
\end{equation}%
\begin{equation}
-\tilde{\lambda}A^{2}+2AB(1-\frac{\tilde{\mu}}{2})+\tilde{\lambda}B^{2}=0%
\text{.}  \label{zetaAB2.2}
\end{equation}

Differentiating the relation (\ref{zetaAB2.2}) with respect to $\xi $ and
using the equations (\ref{zetak}), \ $\xi (\tilde{\mu})=0$, (\ref{eq10.10}),
(\ref{eq11.11}) and (\ref{zetaAB2.2}), we obtain%
\begin{equation}
(A(1-\frac{\tilde{\mu}}{2})+\tilde{\lambda}B)^{2}+(B(1-\frac{\tilde{\mu}}{2}%
)-\tilde{\lambda}A)^{2}=0  \label{zetaAB3.3}
\end{equation}%
From (\ref{zetaAB3.3}), precisely following cases occurs.%
\begin{eqnarray}
\bullet A &=&0\text{ and }B=0,  \label{A} \\
\bullet A &\neq &0\text{ and }\tilde{\lambda}^{2}+(1-\frac{\tilde{\mu}}{2}%
)^{2}=0,  \label{B} \\
\bullet B &\neq &0\text{ and }\tilde{\lambda}^{2}+(1-\frac{\tilde{\mu}}{2}%
)^{2}=0,  \label{C} \\
\bullet A &=&0\text{ and }\tilde{\lambda}^{2}+(1-\frac{\tilde{\mu}}{2}%
)^{2}\neq 0,  \label{D} \\
\bullet B &=&0\text{ and }\tilde{\lambda}^{2}+(1-\frac{\tilde{\mu}}{2}%
)^{2}\neq 0.  \label{E}
\end{eqnarray}%
We now check, case by case, whether (\ref{zetaAB3.3}) give rise to a local
classification of generalized ($\tilde{\kappa}\neq -1,\tilde{\mu})$%
-paracontact metric manifolds with $\tilde{\kappa}<-1$. From (\ref{A}) we
get $\tilde{\kappa}$ and $\tilde{\mu}$ constants. So the manifold returns to
a $(\tilde{\kappa},\tilde{\mu})$-paracontact metric manifold. But we want to
give a local classification for generalized $(\tilde{\kappa},\tilde{\mu})$%
-paracontact metric manifolds. So we omit this case. (\ref{B}) and (\ref{C})
hold if and only if $\tilde{\kappa}=-1$ and $\tilde{\mu}=2$. But this is a
contradiction with $\tilde{\kappa}<-1.$ If we use (\ref{D}) in (\ref%
{zetaAB3.3}) we obtain $B^{2}((1-\frac{\tilde{\mu}}{2})^{2}+\tilde{\lambda}%
^{2})=0$. But the solution of this equation contradicts with the type of
manifold and choosing of $\tilde{\kappa}.$

So we can give following corollary.

\begin{corollary}
There is not exist any generalized ($\tilde{\kappa}<-1,\tilde{\mu})$%
-paracontact metric manifolds which satisfy the condition $\xi (\tilde{\mu}%
)=0$.
\end{corollary}

In the following theorem, we will locally construct generalized $(\tilde{%
\kappa},\tilde{\mu})$-paracontact metric manifolds with $\tilde{\kappa}>-1$
and $\xi (\tilde{\mu})=0$.

\begin{theorem}
\label{important} Let $\tilde{\kappa}$ $:I\subset 
\mathbb{R}
\rightarrow 
\mathbb{R}
$ be a smooth function defined on an open interval $I$, such that $\tilde{%
\kappa}(z)>-1$ for any $z\in I$. Then, we can construct two families of
generalized $(\tilde{\kappa}_{i},\tilde{\mu}_{i})$-paracontact metric
manifolds $(M,\tilde{\varphi}_{i},\xi _{i},\eta _{i},\tilde{g}_{i})$, $i=1,2$%
, in the set $M=%
\mathbb{R}
^{2}\times I\subset 
\mathbb{R}
^{3}$, so that, for any $P(x,y,z)\in M$, the following are valid:%
\begin{eqnarray*}
\tilde{\kappa}_{1}(P) &=&\tilde{\kappa}_{2}(P)=\tilde{\kappa}(z),\text{ \ }
\\
\text{\ }\tilde{\mu}_{1}(P) &=&2(1+\sqrt{1+\tilde{\kappa}(z)})\text{ \ \ and
\ \ }\tilde{\mu}_{2}(P)=2(1-\sqrt{1+\tilde{\kappa}(z)})
\end{eqnarray*}

Each family is determined by two arbitrary smooth functions of one variable.

\begin{proof}
We put $\tilde{\lambda}(z)=\sqrt{1+\tilde{\kappa}(z)}>0$ and \ consider on $M
$ the linearly independent vector fields%
\begin{equation}
\xi _{1}=\frac{\partial }{\partial x},\text{ \ \ }X_{1}=a\frac{\partial }{%
\partial x}+b\frac{\partial }{\partial y}+\frac{\partial }{\partial z}\text{%
and }Y_{1}=\frac{\partial }{\partial y}\text{ },  \label{eq4.12}
\end{equation}%
where $a(x,y,z)=-2y+f(z)$, $b(x,y,z)=-\frac{y}{2}\frac{\tilde{\lambda}%
^{\prime }(z)}{\tilde{\lambda}(z)}-2x\tilde{\lambda}(z)+s(z)$, $f(z),$ $s(z)$
are arbitrary smooth functions of $z$ and $-\tilde{g}_{1}(X_{1},X_{1})=%
\tilde{g}_{1}(Y_{1},Y_{1})=\tilde{g}_{1}(\xi _{1},\xi _{1})=1.$The structure
tensor fields $\eta _{1},\tilde{g}_{1},\tilde{\varphi}_{1}$ are defined by $%
\eta _{1}=$ $dx-(-2y+f(z))dz$, $\tilde{g}_{1}=\left( 
\begin{array}{ccc}
1 & 0 & -a \\ 
0 & 1 & -b \\ 
-a & -b & -1+a^{2}+b^{2}%
\end{array}%
\right) $ and $\tilde{\varphi}_{1}=\left( 
\begin{array}{ccc}
0 & a & -ab \\ 
0 & b & 1-b^{2} \\ 
0 & 1 & -b%
\end{array}%
\right) $, respectively. From (\ref{eq4.12}), we can easily obtain 
\begin{eqnarray}
\lbrack \xi _{1},X_{1}] &=&-2\tilde{\lambda}(z)Y_{1}\text{, \ \ }[\xi
_{1},Y_{1}]=0\text{, \ \ }  \label{4.12a} \\
\lbrack X_{1},Y_{1}] &=&\frac{\tilde{\lambda}^{\prime }(z)}{2\tilde{\lambda}%
(z)}Y_{1}+2\xi _{1}.  \label{4.12b}
\end{eqnarray}%
Since $\eta _{1}\wedge d\eta _{1}=2dx\wedge dy\wedge dz\neq 0$ everywhere on 
$M$, we conclude that $\eta _{1}$ is a contact form. By using just defined $%
\tilde{g}_{1}$ and $\tilde{\varphi}_{1}$, we find\ $\eta _{1}=\tilde{g}%
(.,\xi _{1}),$ $\tilde{\varphi}_{1}X_{1}=Y_{1}$, $\tilde{\varphi}%
_{1}Y_{1}=X_{1}$, $\tilde{\varphi}_{1}\xi _{1}=0$ and $d\eta _{1}(Z,W)=%
\tilde{g}_{1}(Z,\tilde{\varphi}_{1}W)$, $\tilde{g}_{1}(\tilde{\varphi}_{1}Z,%
\tilde{\varphi}_{1}W)=-\tilde{g}_{1}(Z,W)+\eta _{1}(Z)\eta _{1}(W)$ for any $%
Z$, $W\in \Gamma (M)$. Hence $M(\eta _{1},\xi _{1},\tilde{\varphi}_{1},%
\tilde{g}_{1})$ is a paracontact metric manifold. From the well known
Koszul's formula $2\tilde{g}_{1}(\tilde{\nabla}_{Z}W,T)=Z\tilde{g}_{1}(W,T)+W%
\tilde{g}_{1}(T,Z)-T\tilde{g}_{1}(Z,W)-\tilde{g}_{1}(Z,\left[ W,T\right] )+%
\tilde{g}_{1}(W,\left[ T,Z\right] )+\tilde{g}_{1}(T,\left[ Z,W\right] )$ and
(\ref{nablaxi}), we have the following equations%
\begin{equation}
\tilde{\nabla}_{X_{1}}\xi _{1}=(\tilde{\lambda}(z)-1)Y_{1},\text{ \ \ }%
\tilde{\nabla}_{Y_{1}}\xi _{1}=-(1+\tilde{\lambda}(z))X_{1},  \label{eq4.13}
\end{equation}%
\begin{equation}
\tilde{\nabla}_{\xi _{1}}\xi _{1}=0,\text{ \ \ }\tilde{\nabla}_{\xi
_{1}}X_{1}=-(\tilde{\lambda}(z)+1)Y_{1},\text{ \ \ }\tilde{\nabla}_{\xi
_{1}}Y_{1}=-(1+\tilde{\lambda}(z))X_{1},  \label{eq414}
\end{equation}%
\begin{equation}
\tilde{\nabla}_{X_{1}}X_{1}=0,\text{ \ \ }\tilde{\nabla}_{Y_{1}}Y_{1}=\frac{%
\tilde{\lambda}^{\prime }(z)}{2\tilde{\lambda}(z)}X_{1},  \label{eq4.15}
\end{equation}%
\begin{eqnarray}
\tilde{\nabla}_{Y_{1}}X_{1} &=&-\frac{\tilde{\lambda}^{\prime }(z)}{2\tilde{%
\lambda}(z)}Y_{1}-(1+\tilde{\lambda}(z))\xi _{1},\text{ \ \ }  \label{eq4.16}
\\
\tilde{\nabla}_{X_{1}}Y_{1} &=&(-\tilde{\lambda}(z)+1)\xi _{1},
\label{eq4.17}
\end{eqnarray}%
$\tilde{h}_{1}\tilde{\varphi}_{1}X_{1}=-\tilde{\lambda}(z)\tilde{\varphi}%
_{1}X_{1}$ and $\tilde{h}_{1}X_{1}=\tilde{\lambda}(z)X_{1}$, where $\tilde{%
\nabla}$ is Levi-Civita connection of $\tilde{g}_{1}$. By using \ the
relations (\ref{eq4.13})-(\ref{eq4.17}) we obtain%
\begin{eqnarray*}
\tilde{R}(\xi _{1},\xi _{1})\xi _{1} &=&0, \\
\tilde{R}(X_{1},\xi _{1})\xi _{1} &=&\tilde{\kappa}_{1}X_{1}+\tilde{\mu}_{1}%
\tilde{h}_{1}X_{1}, \\
\text{\ }\tilde{R}(Y_{1},\xi _{1})\xi _{1} &=&\tilde{\kappa}_{1}Y_{1}+\tilde{%
\mu}_{1}\tilde{h}_{1}Y_{1}, \\
\tilde{R}(X_{1},X_{1})\xi _{1} &=&0,~\tilde{R}(Y_{1},Y_{1})\xi _{1}=0 \\
\tilde{R}(X_{1,}Y_{1})\xi _{1} &=&0.
\end{eqnarray*}%
From the above relations and by virtue of the linearity of \ the curvature
tensor $\tilde{R}$, we conclude that 
\begin{equation*}
\tilde{R}(Z,W)\xi _{1}=(\tilde{\kappa}_{1}I+\tilde{\mu}_{1}\tilde{h}%
_{1})(\eta _{1}(Z)W-\eta _{1}(W)Z)
\end{equation*}%
for any $Z,W\in \Gamma (M),$ i.e. $(M,\tilde{\varphi}_{1},\xi _{1},\eta _{1},%
\tilde{g}_{1})$ is a generalized $(\tilde{\kappa}_{1},\tilde{\mu}_{1})$%
-paracontact metric manifold with $\xi (\tilde{\mu}_{1})=0$ and thus the
construction of the first \ family is completed. For the second
construction, we consider the vector fields 
\begin{equation}
\xi _{2}=\frac{\partial }{\partial x},\text{ \ \ }X_{2}=\frac{\partial }{%
\partial y}\text{ },  \label{i1}
\end{equation}%
\begin{equation}
Y_{2}=(2y+f(z))\frac{\partial }{\partial x}+(-\frac{y}{2}\frac{\tilde{\lambda%
}(z)}{\tilde{\lambda}(z)}-2x\tilde{\lambda}(z)+s(z))\frac{\partial }{%
\partial y}+\frac{\partial }{\partial z}  \label{i2}
\end{equation}%
and define the tensor fields $\eta _{2},\tilde{g}_{2},\tilde{\varphi}_{2},%
\tilde{h}_{2}$ as follows:%
\begin{equation*}
\eta _{2}=dx-(2y+f(z))dz
\end{equation*}%
\begin{equation*}
\text{\ }\tilde{g}_{2}=\left( 
\begin{array}{ccc}
1 & 0 & -a \\ 
0 & -1 & -b \\ 
-a & -b & 1+a^{2}+b^{2}%
\end{array}%
\right) ,\text{ }\tilde{\varphi}_{2}=\left( 
\begin{array}{ccc}
0 & a & -ab \\ 
0 & b & 1-b^{2} \\ 
0 & 1 & -b%
\end{array}%
\right) ,
\end{equation*}%
\begin{equation*}
\text{\ \ \ \ }\tilde{h}_{2}=\left( 
\begin{array}{ccc}
0 & 0 & -a\tilde{\lambda}_{2} \\ 
0 & \tilde{\lambda}_{2} & -2\tilde{\lambda}_{2}b \\ 
0 & 0 & -\tilde{\lambda}_{2}%
\end{array}%
\right) 
\end{equation*}%
with respect to the basis $\left( \frac{\partial }{\partial x},\frac{%
\partial }{\partial y},\frac{\partial }{\partial z}\right) ,$ where $%
a=2y+f(z)$, $b=(-\frac{y}{2}\frac{\tilde{\lambda}^{^{\prime }}(z)}{\tilde{%
\lambda}(z)}-2x\tilde{\lambda}(z)+s(z))$ and $-\tilde{g}_{2}(X_{2},X_{2})=%
\tilde{g}_{2}(Y_{2},Y_{2})=\tilde{g}_{2}(\xi _{2},\xi _{2})=1$ . As in first
construction, we say that $(M,\tilde{\varphi}_{2},\xi _{2},\eta _{2},\tilde{g%
}_{2})$ is a generalized $(\tilde{\kappa}_{2},\tilde{\mu}_{2})$-paracontact
metric manifold with $\xi (\tilde{\mu}_{2})=0$, where $\tilde{\kappa}_{2}(z)=%
\tilde{\lambda}(z)^{2}-1$ and $\tilde{\mu}_{2}(x,y,z)=2(1-\sqrt{1+\kappa
_{2}(z)})$. This completes the proof of the theorem.
\end{proof}
\end{theorem}

In the following theorem, we give an analytic expression of the scalar
curvature $\tau $ of generalized $(\tilde{\kappa}\neq -1,\tilde{\mu})$%
-paracontact metric manifolds. It is interesting that the same formula holds
both for the case $\tilde{\kappa}<-1$ and $\tilde{\kappa}>-1.$

\begin{theorem}
\label{LAPLAS}Let $(M,\tilde{\varphi},\xi ,\eta ,\tilde{g})$ be a
generalized $(\tilde{\kappa},\tilde{\mu})$-paracontact metric manifold.
Then, 
\begin{equation}
\bigtriangleup \tilde{\lambda}=-X(A)+\tilde{\varphi}X(B)+\frac{1}{2\tilde{%
\lambda}}(A^{2}-B^{2})  \label{m}
\end{equation}%
and 
\begin{equation}
\tau =\frac{1}{\tilde{\lambda}}(\bigtriangleup \tilde{\lambda})-\frac{1}{%
\tilde{\lambda}^{2}}\parallel grad\tilde{\lambda}\parallel ^{2}+2(\tilde{%
\kappa}+\tilde{\mu}),  \label{mm}
\end{equation}%
where $\bigtriangleup \tilde{\lambda}$ is Laplacian of $\tilde{\lambda}$.
\end{theorem}

\begin{proof}
We will give the proof for $\tilde{\kappa}>-1$. The proof for $\tilde{\kappa}%
<-1$  is similar to $\tilde{\kappa}>-1$. Using the definition of the
Laplacian and equations (\ref{zetak}) and (\ref{eq3}) we obtain 
\begin{eqnarray*}
\bigtriangleup \tilde{\lambda} &=&-XX(\tilde{\lambda})+\tilde{\varphi}X%
\tilde{\varphi}X(\tilde{\lambda})+\xi \xi (\tilde{\lambda}) \\
&&+(\tilde{\nabla}_{X}X)\tilde{\lambda}-(\tilde{\nabla}_{\tilde{\varphi}X}%
\tilde{\varphi}X)\tilde{\lambda}-(\tilde{\nabla}_{\xi }\xi )\tilde{\lambda}
\\
&=&-X(A)+\tilde{\varphi}X(B)+\frac{1}{2\tilde{\lambda}}(A^{2}-B^{2}).
\end{eqnarray*}

In order to compute scalar curvature $\tau $ of $M$, we will use (\ref{eq1}%
)-(\ref{eq4}). Defining the curvature tensor $\tilde{R}$, after some
calculations we get%
\begin{eqnarray*}
\tilde{R}(X,\tilde{\varphi}X)\tilde{\varphi}X &=&\tilde{\nabla}_{X}\tilde{%
\nabla}_{\tilde{\varphi}X}\text{ }\tilde{\varphi}X-\tilde{\nabla}_{\tilde{%
\varphi}X}\tilde{\nabla}_{X}\text{ }\tilde{\varphi}X-\tilde{\nabla}_{\left[
X,\tilde{\varphi}X\right] }\tilde{\varphi}X \\
&=&\tilde{\nabla}_{X}\left( -\frac{A}{2\tilde{\lambda}}X\right) -\tilde{%
\nabla}_{\tilde{\varphi}X}\left( -\frac{B}{2\tilde{\lambda}}X+(1-\tilde{%
\lambda})\xi \right) -\tilde{\nabla}_{-\frac{B}{2\tilde{\lambda}}X+\frac{A}{2%
\tilde{\lambda}}\tilde{\varphi}X+2\xi }\tilde{\varphi}X \\
&=&-X\left( \frac{A}{2\tilde{\lambda}}\right) X-\frac{A}{2\tilde{\lambda}}%
\tilde{\nabla}_{X}X+\tilde{\varphi}X\left( \frac{B}{2\tilde{\lambda}}\right)
X+\frac{B}{2\tilde{\lambda}}\tilde{\nabla}_{\tilde{\varphi}X}X \\
&&+\tilde{\varphi}X(\tilde{\lambda})\xi -(1-\tilde{\lambda})\tilde{\nabla}_{%
\tilde{\varphi}X}\xi +\frac{B}{2\tilde{\lambda}}\tilde{\nabla}_{X}\tilde{%
\varphi}X-\frac{A}{2\tilde{\lambda}}\tilde{\nabla}_{\tilde{\varphi}X}\text{ }%
\tilde{\varphi}X-2\tilde{\nabla}_{\xi }\text{ }\tilde{\varphi}X \\
&=&-X\left( \frac{A}{2\tilde{\lambda}}\right) X+\frac{A}{2\tilde{\lambda}}%
\frac{B}{2\tilde{\lambda}}\tilde{\varphi}X+\tilde{\varphi}X\left( \frac{B}{2%
\tilde{\lambda}}\right) X \\
&&+\frac{B}{2\tilde{\lambda}}\left( -\frac{A}{2\tilde{\lambda}}\tilde{\varphi%
}X-(\tilde{\lambda}+1)\xi \right) \\
&&+\tilde{\varphi}X(\tilde{\lambda})\xi +(1+\tilde{\lambda})(1-\tilde{\lambda%
})X \\
&&+\frac{B}{2\tilde{\lambda}}\left( -\frac{B}{2\tilde{\lambda}}X+(1-\tilde{%
\lambda})\xi \right) +\frac{A}{2\tilde{\lambda}}\left( \frac{A}{2\tilde{%
\lambda}}X\right) +2\left( \frac{\tilde{\mu}}{2}X\right) \\
&=&\left[ -X\left( \frac{A}{2\tilde{\lambda}}\right) +\tilde{\varphi}X\left( 
\frac{B}{2\tilde{\lambda}}\right) -\frac{B^{2}}{4\tilde{\lambda}^{2}}+\frac{%
A^{2}}{4\tilde{\lambda}^{2}}+(-\tilde{\lambda}^{2}+1)+\tilde{\mu}\right] X \\
&=&\left[ -\frac{1}{2}\left( \frac{X(A)\tilde{\lambda}-A^{2}}{\tilde{\lambda}%
^{2}}+\frac{\tilde{\varphi}X(B)\tilde{\lambda}-B^{2}}{\tilde{\lambda}^{2}}%
\right) +\frac{1}{4\tilde{\lambda}^{2}}(A^{2}-B^{2})+(-\tilde{\lambda}%
^{2}+1)+\tilde{\mu}\right] X \\
&=&\left[ \frac{1}{2}\frac{-X(A)+\phi X(B)}{\tilde{\lambda}}+\frac{1}{2%
\tilde{\lambda}^{2}}(A^{2}-B^{2})+\frac{1}{4\tilde{\lambda}^{2}}%
(A^{2}-B^{2})+(-\tilde{\lambda}^{2}+1)+\tilde{\mu}\right] X \\
&=&\left[ \frac{1}{2\tilde{\lambda}}\left( -X(A)+\tilde{\varphi}X(B)+\frac{1%
}{2\tilde{\lambda}}(A^{2}-B^{2})\right) +\frac{1}{2\tilde{\lambda}^{2}}%
(A^{2}-B^{2})+(-\tilde{\lambda}^{2}+1)+\tilde{\mu}\right] X \\
&=&\left[ \frac{1}{2\tilde{\lambda}}\bigtriangleup \tilde{\lambda}-\frac{1}{2%
\tilde{\lambda}^{2}}\parallel grad\tilde{\lambda}\parallel ^{2}-\tilde{\kappa%
}+\tilde{\mu}\right] X
\end{eqnarray*}%
and namely%
\begin{equation*}
\tilde{g}(\tilde{R}(X,\tilde{\varphi}X)\tilde{\varphi}X,X)=-\frac{1}{2\tilde{%
\lambda}}\bigtriangleup \tilde{\lambda}+\frac{1}{2\tilde{\lambda}^{2}}%
\parallel grad\tilde{\lambda}\parallel ^{2}+\tilde{\kappa}-\tilde{\mu}.
\end{equation*}

By using definition of scalar curvature, i.e. $\tau =TrQ=-\tilde{g}(QX,X)+%
\tilde{g}(Q\tilde{\varphi}X,\tilde{\varphi}X)+\tilde{g}(Q\xi ,\xi ),$ and
using (\ref{Ricci}), we have%
\begin{eqnarray*}
\tau &=&-2\tilde{g}(\tilde{R}(X,\tilde{\varphi}X)\tilde{\varphi}X,X)+2\tilde{%
g}(Q\xi ,\xi ) \\
&=&\frac{1}{\tilde{\lambda}}\bigtriangleup \tilde{\lambda}-\frac{1}{\tilde{%
\lambda}^{2}}\parallel grad\tilde{\lambda}\parallel ^{2}-2(\tilde{\kappa}-%
\tilde{\mu})+4\tilde{\kappa} \\
&=&\frac{1}{\tilde{\lambda}}\bigtriangleup \tilde{\lambda}-\frac{1}{\tilde{%
\lambda}^{2}}\parallel grad\tilde{\lambda}\parallel ^{2}+2(\tilde{\kappa}+%
\tilde{\mu}).
\end{eqnarray*}%
The last equation gives (\ref{mm}).
\end{proof}

\end{document}